\documentclass[11pt,bezier]{article}
\usepackage{amsmath,amssymb,amsfonts,euscript}
\usepackage{algorithmic, algorithm}
\usepackage{graphicx}

\usepackage[ansinew]{inputenc}
\usepackage{graphicx}
\usepackage{color}
\usepackage{mathrsfs}

\usepackage[colorlinks]{hyperref}
\usepackage[active,new,noold,marker]{xrcs}
\catcode`\=13
\def{$\bowtie$}

\usepackage{eurosym}
\usepackage{lscape} 

\newtheorem{prethm}{{\bf Theorem}}

\newenvironment{thm}{\begin{prethm}{\hspace{-0.5
               em}{\bf}}}{\end{prethm}}

\newtheorem{prepro}{{\bf Theorem}}

\newenvironment{pro}{\begin{prepro}{\hspace{-0.5
               em}{\bf}}}{\end{prepro}}

\newtheorem{preprop}{{\bf Proposition}}

\newtheorem{precor}{{\bf Corollary}}

\newenvironment{cor}{\begin{precor}{\hspace{-0.5
               em}{\bf}}}{\end{precor}}

\newtheorem{preconj}{{\bf Conjecture}}

\newenvironment{conj}{\begin{preconj}{\hspace{-0.5
               em}{\bf}}}{\end{preconj}}

\newtheorem{predefi}{{\bf Definition}}

\newenvironment{defi}{\begin{predefi}{\hspace{-0.5
               em}{\bf}}}{\end{predefi}}

\newtheorem{preali}{{\bf Proof of Theorem 1.}}

\newenvironment{ali}[1]{\begin{preali}{\rm
               #1}\hfill{$\Box$}}{\end{preali}}

\newtheorem{preque}{{\bf Question}}

\newenvironment{que}{\begin{preque}{\hspace{-0.5
               em}{\bf.}}}{\end{preque}}

\newtheorem{preremark}{{\bf Remark}}

\newtheorem{preexample}{{\bf Example}}

\newtheorem{prelem}{{\bf Lemma}}

\newenvironment{lem}{\begin{prelem}{\hspace{-0.5
               em}{\bf}}}{\end{prelem}}

\newtheorem{prelam}{{\bf Lemma}}

\newenvironment{lam}{\begin{prelam}{\hspace{-0.5
               em}{\bf}}}{\end{prelam}}

\newtheorem{preproof}{{\bf Proof}}

\newenvironment{proof}[1]{\begin{preproof}{\rm
               #1}\hfill{$\Box$}}{\end{preproof}}

\title{\bf \large Upper bounds for the 2-hued chromatic number of graphs in terms of the independence number}

\author{{\normalsize {\sc
Ali Dehghan}\, and\, {\sc Arash
Ahadi}}\vspace{3mm}
\\{\footnotesize{\it Department of Mathematical
 Sciences, Sharif
University of Technology.}}
\\{\footnotesize{}}\\{\footnotesize
{$\mathsf{a\_dehghan@mehr.sharif.edu}$\quad\quad
$\mathsf{arash\_ahadi@mehr.sharif.edu}$}}}

\date{}
\begin{document}
\maketitle

\begin{abstract}
{\small \noindent A 2-hued coloring of a graph $G$ (also known as conditional $(k, 2)$-coloring and dynamic coloring) is a coloring such that for every vertex
$v\in V(G)$ of degree at least $2$, the neighbors of $v$ receive at least $2$ colors. The smallest
integer $k$ such that $G$ has a 2-hued coloring with $ k  $ colors, is called the
{\it 2-hued chromatic number} of $G$ and denoted by $\chi_2(G)$. In this paper, we will show that if $G$ is a regular graph, then
$ \chi_{2}(G)- \chi(G) \leq 2 \log _{2}(\alpha(G)) +\mathcal{O}(1)  $ and if $G$ is a graph and $\delta(G)\geq 2$, then
$ \chi_{2}(G)- \chi(G) \leq  1+\lceil  \sqrt[\delta -1]{4\Delta^{2}}  \rceil ( 1+ \log _{\frac{2\Delta(G)}{2\Delta(G)-\delta(G)}} (\alpha(G)) ) $ and in general case if $G$ is a graph, then $  \chi_{2}(G)- \chi(G) \leq 2+ \min \lbrace \alpha^{\prime}(G),\frac{\alpha(G)+\omega(G)}{2}\rbrace $.

}

\begin{flushleft}

\end{flushleft}

\end{abstract}


\section{Introduction}

All graphs in this paper are finite, undirected and simple. We follow the notation and
terminology of \cite{MR1367739}. A {\it vertex coloring} of $G$ is a map $c: V(G) \mapsto \{1,
2, \cdots, k\}$ such that for any adjacent vertices $u$ and $v$ of
$G$, $c(u) \neq c(v)$. As $|c(V(G))| \le k$, $c$ is also called a
$k$-coloring of $G$. We denote a bipartite graph
$G$ with bipartition $(X,Y)$ by $G[X,Y]$. Let $G$ be a graph with a vertex coloring $c$. For every $v\in V (G)$, we denote the degree
of $v$ in $G$, the neighbor set of $v$ and the color of $v$ by $d(v)$, $N(v)$, and $c(v)$, respectively. For any $S\subseteq V(G)$, $N(S)$ denote the set of vertices of $G$, such that each of them has at least one neighbor in $S$.
There are many ways to color the vertices of graphs, an interesting way for vertex coloring
was recently introduced by Lai et al. in \cite{MR1991048}. A vertex $k$-coloring of a graph $G$ is called
{\it 2-hued} if for every vertex $v$ with degree at least $2$, the
neighbors of $v$ receive at least two different colors. The smallest
integer $k$ such that $G$ has a 2-hued $k$-coloring is called the
{\it 2-hued chromatic number} of $G$ and denoted by $\chi_2(G)$.

There exists a generalization for the 2-hued coloring of graphs \cite{MR2251583, Mont}. For an integer $r > 0$, an {\it r-hued k-coloring} of a graph $G$ is a $k$-coloring of the vertices of $G$ such that every vertex $v$ of degree $d(v)$ in $G$ is
adjacent to vertices with at least $\min \lbrace r, d(v) \rbrace$ different colors. The smallest
integer $k$ for which a graph $G$ has an r-hued k-coloring is called the
{\it r-hued k-coloring chromatic number}, denoted by $\chi_r(G)$. An r-hued k-coloring is a generalization of the traditional vertex coloring for which $r = 1$.
The other concept that has a relationship with the 2-hued coloring is the hypergraph coloring. A hypergraph $ H $, is a pair $ (X,Y) $, where $ X $ is the set of vertices and $ Y $ is a set of non-empty subsets of $ X $, called edges. The coloring of $ H $ is a coloring of $ X $ such that for every edge $ e $ with $ \vert e \vert >1 $, there exist $ v,u\in X $ such that $ c(u)\neq c(v) $. For the hypergraph $ H=(X,Y) $, consider the bipartite graph $ \widehat{H} $ with two parts $ X $ and $ Y $, that $ v\in X $ is adjacent to $ e\in Y $ if and only if $ v\in e $ in $ H $. Now consider a 2-hued coloring $ c $ of $ \widehat{H} $, clearly by inducing $ c $ on $ X $, we obtain a coloring of $ H $.
The graph $G^{\frac{1}{2}}$ is said to be the $2$-subdivision of a graph
$G$ if $G^{\frac{1}{2}}$ is obtained from $G$ by replacing each edge with a path with exactly one inner
vertices \cite{MR2519165}. There exists a relationship between $ \chi(G)$ and $ \chi_{2}(G^{\frac{1}{2}}) $. We have
$\chi(G) \leq  \chi_{2}(G^{\frac{1}{2}}) $ and $  \chi(G^{\frac{1}{2}})=2 $. For example it was shown
in \cite{Mont} if $ G\cong K_{n} $ then $ \chi_{2}(K_{n}^{\frac{1}{2}}) \geq n $.
Therefore, there are some graphs such that the difference between the chromatic number and the 2-hued chromatic number
can be arbitrarily large. It seems that when $ \Delta (G) $ is close to $ \delta(G) $, then $ \chi_{2}(G) $ is also close to $ \chi(G) $. Montgomery conjectured that for regular graphs the
difference is at most $2$.

\begin{conj}
\noindent {\bf [Montgomery \rm \cite{Mont}\bf]} For any $r$-regular graph $G$, $\chi_2(G)-\chi(G)\leq 2$.
\end{conj}

Some properties of 2-hued coloring were studied in \cite{akbari2, product, aa1, aa2, MR2251583}.
In \cite{MR2483491}, it has been proved that the computational complexity of $\chi_2(G)$ for a $3$-regular graph is an {\bf NP}-complete problem.
Furthermore, in \cite{complex2} it is shown that it is  {\bf NP}-complete to determine whether there exists a 2-hued coloring with 3 colors for a claw-free graph
with the maximum degree $3$. In \cite{strongly} it was proved that if $G$ is a strongly regular graph and $G \neq C_{4},C_{5},K_{r,r}$, then
$\chi_2(G)-\chi(G)\leq 1$.
Finding the optimal upper bound for $ \chi_{2}(G)-\chi(G) $ seems to be an intriguing problem. In this paper we will prove various inequalities relating it to other graph parameters.

Here, we state some definitions and lemmas that will be used in the sequel of the paper.

\begin{defi}
Let $c$ be a vertex coloring of a graph $G$, then $B_c=\{v\in V(G)\mid d(v)\geq 2,\, |c(N(v))|=1\}$, also every vertex in $B_c$ is called  a
{\it bad vertex} and every vertex in $V(G)\setminus B_c=A_c$ is called a {\it good vertex}.
\end{defi}

For every graph $G$ define,

\begin{center}
$k^*(G)=
\begin{cases}
   2,      &if\,\,\chi(G)=2\\
   1,     & if\,\,\chi(G) \in \lbrace 3,4,5\rbrace\\
   0,       & otherwise.\
\end{cases}$
\end{center}

For the simplicity we denote   $k^*(G)$  by $k^*$. In \cite{aa3} it was proved that for every graph $G$, there exists a vertex coloring with at most $\chi(G)+2$ colors such that the set of bad vertices is independent.

\begin{pro}\label{Th.A}
{\rm\cite{aa3}}
{Let $G$ be a graph. Then there exists a vertex $(\chi(G)+
k^*)$-coloring of $G$ such that the set of bad vertices of $G$ is
independent.}
\end{pro}

The most important bound
for $\chi_{2}(G)$ is the following theorem:

\begin{pro}\label{Th.B}
{\rm\cite{MR1991048}} For a connected graph $G$ if $\Delta(G)\leq 3$, then $\chi_{2}(G) \leq 4$ unless
$G = C_{5}$, in which case $\chi_{2}(C_{5}) = 5$ and if $\Delta(G) \geq 4$, then $\chi_{2}(G) \leq \Delta(G)+ 1$.
\end{pro}

We will use the probabilistic method to prove Theorem \ref{Th.1}.

\begin{lam}\label{L.A}
\noindent {\bf [The Lovasz Local Lemma \rm \cite{MR1885388}\bf]} Suppose ${A_{1},\ldots ,A_{n}}$ is a set of random events such that for each
 $i$, $Pr(A_{i})\leq p$ and $A_{i}$ is mutually independent of the set of all but at most $d$ other events. If $4pd\leq 1$, then with positive probability, none of the events occur.
\end{lam}

\begin{lam}\label{L.B}
\noindent {\rm \cite{Akbari}} Let $r \geq 4$ be a natural number. Suppose that $G [A,B]$ is a bipartite graph such
that all vertices of Part $A$ have degree $r$ and all vertices of Part $B$ have degree at most $r$. Then
one can color the vertices of Part $B$ with two colors such that every vertex $v$ of Part $A$, with $d(v)\geq 2$ receives at least two colors in its neighbors.
\end{lam}

\begin{lam}\label{L.C}
\noindent {\rm \cite{MR1367739}} A set of vertices in a graph is an independent dominating set if and only if it is a maximal independent set.
\end{lam}

\begin{defi}
Let $G$ be a graph and $T_1, T_2 \subseteq  V(G)$, then $T_{1}$ is a dominating set for $T_{2} $ if and only if, for every vertex $v\in T_2$, not in $T_1$, is joined to at least one vertex of $T_1$.
\end{defi}

\section{Main Results}
\label{}

\begin{thm}\label{Th.1}
\label{t1} If $G$ is a graph and $\delta(G) \geq 2$, then $ \chi_{2}(G)- \chi(G) \leq  \lceil \sqrt[\delta -1]{4\Delta^{2}}  \rceil (\lfloor \log _{\frac{2\Delta(G)}{2\Delta(G)-\delta(G)}} (\alpha(G) )\rfloor +1)+1 $.
\end{thm}

Before proving our main theorem we need to prove some lemmas.

\begin{lem}\label{L.1}
If $G$ is a graph, $G\neq \overline{K_{n}}$ and $T_{1}$ is an independent set of $G$, then there exists $T_{2}$ such that, $T_{2}$ is an independent dominating set for $T_{1}$ and
$ \vert T_{1} \cap T_{2}\vert \leq \frac{2\Delta(G) -\delta(G) }{2\Delta(G)} \vert T_{1}\vert$.
\end{lem}

\begin{proof}{

The proof is constructive. In order to find $T_2$, perform Algorithm 1.
When Algorithm 1 terminates, because of Step $3$ and Step $5$, $ T_{4} $ is an independent set and because of Step $3$ and Step $6$, $ T_{4} $ is a dominating set for $ T_{1} \backslash T_{3} $. Now let $T_{2}=T_{4}\cup T_{3}$, because of Step $6$, $ T_{2} $ is an independent dominating set for $T_{1}$.
Assume that Algorithm 1 has $l$ iterations. Because of Step $6$, we have $ s=\sum_{i=1}^{l} t_{i} $.
 Each vertex in $N(T_{1})$ has at most $\Delta(G)-1 $ neighbors in $N(T_{1}) $, so in Step $4$ of the $i$th iteration, $\sum_{u\in N(T_{1})} f(u)$ is decreased at most $t_{i} \Delta(G)$ and in Step $5$ of the $i$th iteration, $\sum_{u\in N(T_{1})} f(u)$ is decreased at most $ t_{i} \Delta(G)$, so in the $i$th iteration, $\sum_{u\in N(T_{1})} f(u)$ is decreased at most $2 t_{i} \Delta(G)$. When Algorithm 1 terminates, $\sum_{ u\in N(T_{1})} f(u)=0 $, so:

$ \delta(G)\vert T_{1}\vert - \displaystyle\sum_{ i=1 }^{ i=l } (2t_{i}\Delta(G)) \leq 0$,

 $ \delta(G)\vert T_{1}\vert - 2s\Delta(G) \leq 0$,

 $  s\geq \frac{\delta(G)}{2\Delta(G)} \vert T_{1}\vert$ ,

$\vert T_{1} \cap T_{2}\vert =\vert T_{3}\vert = \vert T_{1}\vert - s \leq  \frac{2\Delta(G) -\delta(G) }{2\Delta(G)} \vert T_{1}\vert$.

}\end{proof}

\begin{algorithm}
\label{ZZ}
\caption{}

\begin{algorithmic}

\STATE\noindent {\bf Step 1.} For each $u\in N(T_{1})$, define the variable $ f(u) $ as the number of vertices which are adjacent to $u$ and are in $T_{1}$. $\sum_{u\in N(T_{1})} f(u)$ is the number of edges of $G[T_{1},N(T_{1})]$, so $\sum_{u\in N(T_{1})} f(u)\geq \vert T_{1}\vert\delta(G)$.

\STATE\vspace*{.05 cm}

\STATE\noindent {\bf Step 2.} Let $T_{3}=T_{1}$, $T_{4}=\emptyset $, $s=0$, $i=1$.

\STATE\vspace*{.05 cm}

\STATE\noindent {\bf Step 3.} Select a vertex $u$ such that $f(u)$ is maximum among $ \lbrace f(v) \vert v\in N(T_{1})\rbrace$ and add $u$ to the set $T_{4}$ and let $t_{i}=f(u)$.

\STATE\vspace*{.05 cm}

\STATE\noindent {\bf Step 4.} For each $v\in N(T_{1})$ that is adjacent to $u$, change the value of $ f(v) $ to $0$. Change the value of $ f(u) $ to $0$.

\STATE\vspace*{.05 cm}

\STATE\noindent {\bf Step 5.} For each $v\in N(T_{1})$ that is adjacent to at least one vertex of $N(u)\cap T_{3}$ and it is not adjacent to $u$, decrease $ f(v)$ by the number of common neighbors of $v$ and $u$ in $T_{3}$.

\STATE\vspace*{.05 cm}

\STATE\noindent {\bf Step 6.} Remove the elements of $N(u)$ from $T_{3}$. Increase $s$ by $t_{i}$ and $i$ by $1$.

\STATE\vspace*{.05 cm}

\STATE\noindent {\bf Step 7.} If $\sum_{u\in N(T_{1})} f(u)>0$ go to Step 3.

\end{algorithmic}
\end{algorithm}

\begin{lem}\label{L.2}
If $G$ is a graph, $ \delta \geq 2 $ and $T$ is an independent set of $G$, then we can color the vertices of $T$ with $ \lceil (4\Delta^{2})^{\frac{1}{\delta-1}} \rceil $ colors such that for each $u\in \lbrace v \vert v\in V(G), N(v)\subseteq T\rbrace$, $N(u)$ has at least two different colors.
\end{lem}

\begin{proof}{
Let $ \eta = \lceil (4\Delta^{2})^{\frac{1}{\delta-1}} \rceil  $. Color every vertex of $T$ randomly and independently by one color from $ \lbrace 1,\cdots,\eta \rbrace $, with the same probability.
For each $u\in \lbrace v \vert v\in V(G), N(v)\subseteq T\rbrace$, let $A_{u}$ be the event that all
of the neighbors of $ u $ have a same color.
Each $A_{u}$ is mutually independent of a set of all $A_{v}$ events but at most $\Delta ^{2}$ of them.
Clearly, $Pr(A_{u})\leq \frac{1}{\eta^{\delta-1}}$. We have: $4pd= 4(\frac{1}{\eta})^{\delta-1} \Delta ^{2}\leq 1$.
So by Local Lemma there exists a coloring with our condition for $T$ with positive probability.

}\end{proof}

\begin{lem}\label{L.3}
Let $c$ be a vertex $k$-coloring of a graph $G$, then there exists a 2-hued coloring of $G$ with at most
$k+\vert B_{c}\vert$ colors.
\end{lem}

\begin{algorithm}
\label{A2}
\caption{}

\begin{algorithmic}

\STATE \noindent {\bf Step 1.}  Let $T'_{1}=T_{1}$.

\STATE\vspace*{.05 cm}

\FOR{$i = 2$ to $i=k$}

\STATE\vspace*{.05 cm}

\STATE  \noindent {\bf Step 2.} By Algorithm 1, find an independent set $T_{i}$ such that, $T_{i}$ is an independent dominating set for $T'_{i-1}$ and $\vert T_{i} \cap T'_{i-1} \vert\leq \frac{2\Delta(G) -\delta(G)} {2\Delta(G) } \vert T'_{i-1}\vert$.

\STATE\vspace*{.05 cm}

\STATE \noindent {\bf Step 3.} By Lemma \ref{L.2}, recolor the vertices of $T_{i}$ with the colors $\chi+\eta i-(\eta -1),\ldots, \chi+\eta i $ such that for each $u\in \lbrace v \vert v\in V(G), N(v)\subseteq T_{i}\rbrace$, $N(u)$ has at least two different colors.

\STATE\vspace*{.05 cm}

\STATE \noindent {\bf Step 4.} Let $ T'_{i}=T_{i} \cap T'_{i-1} $.

\STATE\vspace*{.05 cm}

\ENDFOR

\end{algorithmic}
\end{algorithm}

\begin{ali}{
Let $ \eta = \lceil \sqrt[\delta -1]{4\Delta^{2}}  \rceil  $ and $k=\lfloor \log _{\frac{2\Delta(G)}{2\Delta(G)-\delta(G)}} \alpha(G) \rfloor +1$. By Lemma \ref{L.C}, let $T_{1}$ be an independent dominating set for $G$. Consider a vertex $\chi(G)$-coloring of $G$, by Lemma \ref{L.2}, recolor the vertices of $T_{1}$ by the colors $\chi +1,\ldots,\chi+\eta$ such that for each $u\in \lbrace v \vert v\in V(G), N(v)\subseteq T_{1}\rbrace$, $N(u)$ has at least two different colors. Therefore we obtain a coloring $c_{1}$ such that $B_{c_{1}}\subseteq T_{1}$. Now, perform Algorithm \ref{A2}.
After each iteration of Algorithm \ref{A2}
we obtain a coloring $c_{i}$ such that $B_{c_{i}}\subseteq T'_{i}$, so when the procedure terminates, we have a coloring $c_{k}$ with at most $\chi(G) +\eta k $ colors, such that $B_{c_{k}}\subseteq T'_{k}$ and $\vert T'_{k}\vert \leq 1$, so by Lemma \ref{L.3} we have a 2-hued coloring  with at most $\chi(G) +\eta k +1 $ colors.

}\end{ali}

\begin{cor}
if $G$ is a graph and $\Delta(G)\leq   2^{\frac{\delta(G) -3}{2}}$, then
$ \chi_{2}(G)- \chi(G) \leq  2\lfloor \log _{\frac{2\Delta(G)}{2\Delta(G)-\delta(G)}} (\alpha(G)) \rfloor +3 $.
\end{cor}

\begin{thm}\label{Th.2}
If $G$ is a regular graph, then
$ \chi_{2}(G)- \chi(G) \leq  2\lfloor \log _{2} (\alpha(G))\rfloor +3$.
\end{thm}

\begin{proof}{
If $r=0$, then the theorem is obvious. For $1\leq r\leq 3$, we have $\chi(G)\geq 2$, by Theorem \ref{Th.B}, $ \chi_{2}(G)\leq 5 $ so $\chi_{2}(G)\leq \chi(G) +3$.
So assume that $ r \geq 4 $, we use a proof similar to the proof of Theorem \ref{Th.1}. In the proof of Theorem \ref{Th.1}, for each $i$, $ 1\leq i \leq k $ we used Lemma \ref{L.2}, to recolor the vertices of $T_{i}$ with the colors $\chi+\eta i-(\eta -1),\ldots, \chi+\eta i $ such that for each $u\in \lbrace v \vert v\in V(G), N(v)\subseteq T_{i}\rbrace$, $N(u)$ has at least two different colors. In the new proof, for each $i$, $ 1\leq i \leq k $, let $ A=  \lbrace v \vert v\in V(G), N(v)\subseteq T_{i}\rbrace$ and $B=T_{i}$ and by Lemma \ref{L.B}, recolor the vertices of $T_{i}$ with the colors $ \chi+2i-1 $ and $ \chi+2i $ such that for each $u\in \lbrace v \vert v\in V(G), N(v)\subseteq T_{i}\rbrace$, $N(u)$ has at least two different colors. The other parts of the proof are similar. This completes the proof.
 }\end{proof}


\begin{thm}\label{Th.3}
If $G$ is a simple graph, then $ \chi_{2}(G)-\chi(G) \leq \frac{\alpha(G)+\omega(G) }{2} +k^* $.
\end{thm}
\begin{proof}{

For $\chi(G)=1 $ the theorem is obvious.
Suppose that $G$ is a connected graph with $ \chi(G) \geq 2 $, otherwise we apply the following proof for each of its connectivity components.
By Theorem \ref{Th.A}, suppose that $ c $ is a vertex $(\chi(G)+k^* )$-coloring of $G$ such that $ B_{c} $ is an independent set. Also, let $ T_{1} $ be a maximal independent set that contains $ B_{c} $. Consider the partition $ \lbrace \lbrace v_{1},v_{2} \rbrace ,\ldots \lbrace v_{2s-1},v_{2s}\rbrace ,T_{2}=\lbrace v_{2s+1},\ldots ,v_{l} \rbrace \rbrace $ for the vertices of $T_{1}  $ such that for each $i$, $ 1\leq i \leq s $, $ N(v_{2i-1})\cap N(v_{2i})\neq \emptyset$ and for every $i$ and $j$, $ 2s < i <j \leq l $, $ N(v_{i})\cap N(v_{j})=\emptyset $.
For every $i$, $ 1\leq i \leq s $, let $ w_{i}\in N(v_{2i-1})\cap N(v_{2i})$, recolor $ w_{i} $ by the color $ \chi+k^*+i $. name the resulted coloring $c^{\prime}$.
Now, consider the partition $ \lbrace\lbrace  v_{2s+1},v_{2s+2} \rbrace , \ldots ,\lbrace v_{2t-1},v_{2t} \rbrace , T_{3}=\lbrace v_{2t+1},\ldots ,v_{l} \rbrace \rbrace   $ for the vertices of $ T_{2} $ such that for $i$, $ s < i \leq t $ there exist $ u_{2i-1}\in N(v_{2i-1}) $ and $ u_{2i}\in N(v_{2i}) $, such that $ u_{2i-1}$ and $u_{2i} $ are not adjacent and for $i$ and $j$, $ 2t < i <j \leq l $, every neighbor of $v_{i}  $ is adjacent to every neighbor of $v_{j}  $.
For every $i$, $ s < i \leq t $, suppose that $ u_{2i-1}\in N(v_{2i-1}) $ and $ u_{2i}\in N(v_{2i}) $ such that $ u_{2i-1}$ and $u_{2i}  $ are not adjacent. Now if $ c(u_{2i-1}) \neq c^{\prime}(u_{2i-1}) $, then recolor $u_{2i-1}$ by the color $ \chi+k^*+i $ and also if $ c(u_{2i}) \neq c^{\prime}(u_{2i}) $, then recolor $u_{2i}$ by the color $ \chi+k^*+i $.

After above procedure we obtain a coloring, name it $ c^{\prime\prime} $. If $ z $ is a vertex with $ N(z)=\lbrace u_{2i-1},u_{2i} \rbrace $ for some $i$, $ s < i \leq t $ and $  c^{\prime\prime}(u_{2i-1})=c^{\prime\prime}(u_{2i} )$,
therefore  $  c(u_{2i})=c^{\prime}(u_{2i} )$ and $ z\in T_{1} $. Since $u_{2i}  $ is a common neighbor of $ v_{2i} $ and $ z $, therefore $\lbrace z , v_{2i-1} \rbrace \in T_{1}\setminus T_{2}$. It is a contradiction.
For $ v_{i}\in T_{3} $ let $ x_{i}\in N(v_{i}) $. Suppose that $X= \lbrace x_{i} \vert v_{i}\in T_{3} \rbrace $. The vertices of $X  $ make a clique, recolor $X$ by different new colors. We have $ \vert X \vert =l-2t\leq \omega(G) $. Therefore:

\begin{center}
$\chi_{2}(G)-\chi(G) \leq s+(t-s)+(l-2t)+k^* \leq \frac{\alpha(G) +\omega(G)}{2}+k^*$.
\end{center}

}\end{proof}

\begin{cor}
If $G$ is a triangle-free graph, then $ \chi_{2}(G)-\chi(G) \leq \frac{\alpha(G) }{2} +1+k^*  $
\end{cor}

If $G$ is an $r$-regular graph and $ r > \frac{n}{2} $, then every vertex $ v\in V(G) $ appears in some triangles,
therefore $ \chi_{2}(G)=\chi(G)$. In the next theorem, we present an upper bound for the 2-hued chromatic
number of $r$-regular graph $G$ with $ r \geq \frac{n}{k} $ in terms of $ n $ and $ r $.

\begin{thm}\label{Th.4}
If $G$ is an $r$-regular graph with $ n $ vertices, then $ \chi_{2}(G)-\chi(G) \leq 2 \lceil \frac{n}{r} \rceil -2$.
\end{thm}

\begin{proof}{
If $ r\leq 2 $, then the theorem is obvious. If $ r=3 $, then $n\geq 4$, therefore by Theorem \ref{Th.B} the theorem is clear.
Therefore suppose that $ r\geq 4 $ and $ c $ is a vertex $ \chi(G) $-coloring of $G$. For every $k$, $ 1 \leq k  \leq \lceil \frac{n}{r}  \rceil -1 $, let $ T_{k} $ be a maximum independent set of $ G  \setminus \cup _{i=1}^{k-1}T_{i}$.
By Lemma \ref{L.B}, recolor the vertices of $T_{1}$ with two new colors, such that for each $u\in \lbrace v \vert v\in V(G), N(v)\subseteq T_{1}\rbrace$, $u$ has two different colors in $N(u)$. Therefore $G$ has the coloring $ c^{\prime} $ by $ \chi(G) +2 $ colors such that $ B_{c^{\prime}}\subseteq T_{1} $.
Also by Lemma \ref{L.B} recolor every $ T_{k} $ $(  2 \leq k  \leq \lceil \frac{n}{r}  \rceil -1 ) $, by two different new colors. Thus,
$ G $ has a coloring $ c^{\prime\prime} $ such that for every vertex $ v\in V(G) $ with $ N(v)\subseteq T_{k} $, for some $ k $, $ v $ has at least two different colors in its neighbors. We claim that $ c^{\prime\prime} $ is a 2-hued coloring, otherwise suppose that $ u\in B_{c^{\prime\prime}} $. We have $ u\in T_{1}  $ so  $ N(u) $ is an independent set and $ N(u) \cap ( \cup_{i=1}^{\lceil \frac{n}{r} \rceil -1} T_{k})= \emptyset$. Consider the definitions of $ T_{k} $, $  1 \leq k  \leq \lceil \frac{n}{r}  \rceil -1 $, we have:

$ r= \vert N(u) \vert \leq \vert T_{\lceil \frac{n}{r}  \rceil-1}\vert \leq \vert T_{\lceil \frac{n}{r}  \rceil-2}\vert \leq \ldots \leq \vert T_{2}\vert  $,

and $ \vert T_{2}\vert  < \vert T_{1}\vert  $, since otherwise $ T_{2} \cup \lbrace u\rbrace $ is an independent set and $ \vert T_{2} \cup \lbrace u \rbrace \vert > \vert T_{1} \vert$.

Therefore $ n \geq r \lceil \frac{n}{r}  \rceil +1 $, but it is a contradiction. This completes the proof.
}\end{proof}

\begin{thm}\label{Th.5}
If $G$ is a simple graph, then $ \chi_{2}(G)-\chi(G) \leq \alpha^{\prime}(G) +k^*  $.
\end{thm}
\begin{proof}{
Let $G$ be a simple graph. By Theorem \ref{Th.A}, suppose that $c$ is a vertex $ (\chi(G)+k^* )$-coloring of $G$ such that $ B_{c} $ is an independent set. Let $ M=\lbrace v_{1}u_{1},\ldots,v_{\alpha^{\prime}}u_{\alpha^{\prime}} \rbrace $ be a maximum matching of $G$ and $ W=\lbrace v_{1},u_{1},\ldots,v_{\alpha^{\prime}},u_{\alpha^{\prime}} \rbrace $. Let $ X=B_{c} \cap W $ and $ Y= \lbrace v_{i}\vert u_{i}\in X \rbrace \cup  \lbrace u_{i}\vert v_{i}\in X \rbrace$. Recolor the vertices of $Y$ by different new colors. Also recolor every vertex in $ N(B_{c}\setminus X) \cap W $, by a different new color. Call this coloring $ c^{\prime} $. Clearly, $ c^{\prime} $ is a 2-hued coloring of $G$. In order to complete the proof, it is enough to show that we used at most $\alpha^{\prime}(G)$ new colors in $ c^{\prime} $.
If $ e=v_{i}u_{i}  $ $ (1 \leq i \leq \alpha^{\prime}(G) ) $ is an edge of $ M $ such that $ c(v_{i})\neq c^{\prime}(v_{i}) $ and $ c(u_{i})\neq c^{\prime}(u_{i}) $, then three cases can be considered:

$\bullet$ $ v_{i},u_{i}\in Y $. It means that $ \lbrace v_{i},u_{i} \rbrace \subseteq B_{c} $. Therefore $ v_{i} $ and $ u_{i} $ are adjacent, but $ B_{c} $ is an independent set.

$\bullet$ $( v_{i}\in Y$ and $u_{i}\notin Y )$ or $( u_{i}\in Y$ and $v_{i}\notin Y )$. Without loss of generality suppose that  $v_{i}\in Y$ and $u_{i}\notin Y$. So $ u_{i}\in X $ and $ v_{i}\notin X $, therefore there exists $ u^{\prime} \in B_{c} $ such that $u^{\prime} u_{i} \in E(G)  $, but $ B_{c} $ is an independent set.

$\bullet$ $ v_{i},u_{i}\notin Y $. It means that $ v_{i},u_{i}\notin X $ and there exist $ v^{\prime},u^{\prime}\in B_{c} $ such that $ v^{\prime}v_{i},u^{\prime}u_{i}\in E(G)$. Now $ M^{\prime}= (M \setminus \lbrace v_{i}u_{i}\rbrace ) \cup \lbrace v^{\prime}v_{i},u^{\prime}u_{i}\rbrace$ is a matching that is greater than $ M $.

Therefore we recolor at most one of the $ v_{i}$ and $u_{i} $ for each $i$, $ 1\leq i \leq \alpha^{\prime}(G) $, this completes the proof.

}\end{proof}


\section{Remarks}
\label{}

 In Lemma \ref{L.1} we proved if $T_{1}$ is an independent set for a graph $G$ then, there exists $T_{2}$ such that,
 $T_{2}$ is an independent dominating set for $T_{1}$ and
$ \vert T_{1} \cap T_{2}\vert \leq \frac{2\Delta(G) -\delta(G) }{2\Delta(G)} \vert T_{1}\vert$. Finding the optimal upper bound for $ \vert T_{1} \cap T_{2}\vert$ seems to be an intriguing open problem. Here, we ask the following question.

\begin{que}\label{Q1}
Suppose that $G$ is an $r$-regular graph with $ r\neq 0 $. If $T_{1}$ is an independent set, is there exist an independent dominating set $T_{2}$ for $T_{1}$ such that $  T_{1} \cap T_{2} =\emptyset $?
\end{que}

If Question \ref{Q1} is true, it is easy to see that, for every regular graph $G$, we have $\chi_{2}(G)-\chi(G) \leq 4$.


\bibliographystyle{plain}
\bibliography{Refff}

\end{document}